\documentclass[runningheads]{llncs}

\usepackage{amssymb}
\usepackage{empheq}
\usepackage[mathscr]{euscript}
\usepackage{stmaryrd}
\usepackage{mathabx}
\usepackage{dsfont}
\usepackage{graphicx}
\usepackage{multirow}
\usepackage{algorithm2e}
\usepackage[table]{xcolor}
\usepackage{color}
\usepackage{tikz}
\usetikzlibrary{3d}
\usetikzlibrary{decorations.pathreplacing,angles,quotes}
\usepackage{pgfplots}
\usepackage{algorithm2e}

\usepackage[colorlinks=true,breaklinks=true,bookmarks=true,urlcolor=blue,
     citecolor=blue,linkcolor=blue,bookmarksopen=false,draft=false]{hyperref}

\usepackage{braket}

\newcommand{\calA}{\mathcal{A}}

\newcommand{\calL}{\mathcal{L}}

\newcommand{\calT}{\mathcal{T}}

\newcommand{\scrO}{\mathscr{O}}

\newcommand{\bbN}{\mathbb{N}}
\newcommand{\bbR}{\mathbb{R}}

\newcommand{\bbZ}{\mathbb{Z}}

\newcommand{\aff}{\operatorname{aff}}

\newcommand{\Conv}{\operatorname{Conv}}
\newcommand{\Em}{\operatorname{Em}}
\newcommand{\ext}{\operatorname{ext}}

\newcommand{\Proj}{\operatorname{Proj}}
\newcommand{\relint}{\operatorname{relint}}
\newcommand{\rev}{\operatorname{rev}}
\newcommand{\Span}{\operatorname{span}}

\newcommand\defeq{\mathrel{\overset{\makebox[0pt]{\mbox{\normalfont\tiny\sffamily def}}}{=}}}

\newcommand{\sidx}[1]{\left\llbracket     #1 \right\rrbracket}
\renewcommand{\int}{\operatorname{int}}

\newcommand{\epi}{\operatorname{epi}}

\newcommand{\HRB}{\ensuremath{H^{\operatorname{br}}}}
\newcommand{\HZZ}{\ensuremath{H^{\operatorname{zz}}}}

\newcommand{\blue}[1]{{#1}}

\begin{document}

\title{A geometric way to build strong mixed-integer programming formulations}


\author{Joey Huchette\inst{1} \and Juan Pablo Vielma\inst{2}}

\institute{Rice University,~\email{joehuchette@rice.edu} \and MIT,~\email{jvielma@mit.edu}}

\maketitle

\begin{abstract}
    We give an explicit geometric way to build mixed-integer programming (MIP) formulations for unions of polyhedra. The construction is simply described in terms of spanning hyperplanes in an $r$-dimensional linear space. The resulting MIP formulation is ideal, and uses exactly $r$ integer variables and $2 \times (\# \text{ of spanning hyperplanes})$ general inequality constraints. We use this result to derive novel logarithmic-sized ideal MIP formulations for discontinuous piecewise linear functions and structures appearing in robotics and power systems problems.
\end{abstract}


\section{Introduction} \label{sec:introduction}

Consider a disjunctive set $\bigcup_{i=1}^d S^i$, where each \emph{alternative} $S^i \subset \bbR^n$ is a rational polyhedra. Disjunctive constraints of the form $x \in \bigcup_{i=1}^d S^i$ abound in optimization: they are useful, for example, to model nonlinearities~\cite{Geisler:2012,Misener:2012} or discrete logic imposed by complex processes~\cite{Bergamini:2005,Trespalacios:2014b}. Therefore, we would like a way to represent these constraints in such a way that we can efficiently optimize over them. Additionally, we would like to do this in a composable way, as disjunctive constraints frequently arise as substructures in large, complex optimization problems.

Mixed-integer programming (MIP) offers one such solution. MIP formulations are useful because there are sophisticated algorithms---and corresponding high-quality software implementations---that can optimize over these representations efficiently in practice~\cite{Bixby:2007,Junger:2010}. Furthermore, combining MIP formulations for different substructures is trivial, and so this technology can be marshalled for very complex and large-scale optimization problems.

Our contribution in this work is an \emph{explicit geometric construction for strong MIP formulations of disjunctive sets.} In particular, we give a constructive geometric description for the convex hull of a particularly structured set corresponding to the ``embedding'' of a disjunctive set in a higher-dimensional space. By carefully choosing how we do this embedding, we will be able to apply the main technical result to construct small, strong MIP formulations for certain disjunctive constraints of broad interest.

\section{Preliminaries}

Although predicting the relative performance of different MIP formulations is a difficult task, two properties that tend to correlate strongly with good computational performance are \emph{strength} and \emph{small size}. A MIP formulation for $x \in \bigcup_{i=1}^d S^i \subseteq \bbR^n$ is given by its LP relaxation $R \subseteq \bbR^{n+r}$, where $\Proj_x(R \cap (\bbR^n \times \bbZ^r)) = \bigcup_{i=1}^d S^i$ is the orthogonal projection onto the $x$ variables. In this work, we focus exclusively on \emph{ideal formulations}, where the extreme points of the LP relaxation are naturally integral: $\ext(R) \subseteq \bbR^{n} \times \bbZ^{r}$. Ideal formulations are the strongest possible formulations in terms of the LP relaxation, and tend to perform very well computationally (see~\cite{Vielma:2015} for a more detailed discussion). Additionally, we will seek formulations that are small, requiring few additional variables and constraints. More concretely, we will endeavor to build formulations with few integer variables $r$ (ideally, logarithmic in $d$), and few \emph{general inequality} constraints (i.e. those not equivalent to variable bounds). \blue{We focus on general inequality constraints with since modern MIP solvers can incorporate variable bounds directly with very little overhead cost.}

\subsection{Combinatorial disjunctive constraints}

We will focus on the subclass of disjunctive constraints known as \emph{combinatorial disjunctive constraints}~\cite{Huchette:2016a}. In particular, we assume that each alternative is a face of the unit simplex $\Delta^V = \Set{\lambda \in \bbR^V_{\geq 0} | \sum_{v \in V}\lambda_v = 1}$, where $\bbR_{\geq 0}^V$ is the nonnegative orthant in $\bbR^V$. Formally, this means we consider disjunctive constraints of the form $\bigcup_{i=1}^d P(T^i)$ for some finite sets $T^i \subseteq V$, where $P(T) \defeq \Set{\lambda \in \Delta^V | \lambda_v \neq 0 \text{ only if } v \in T}$. We will use $\calT = (T^i)_{i=1}^d$ notationally as the collection of subsets corresponding to each alternative.

This class of disjunctive constraints is far more expressive than it may appear at first glance. Suppose you have an arbitrary disjunctive constraint $\bigcup_{i=1}^d S^i$, presuming for the moment that each $S^i$ is bounded. If you take each $T^i = \ext(S^i)$ as the set of extreme points of $S^i$, then
\[
    \bigcup\nolimits_{i=1}^d S^i = \Set{ \sum\nolimits_{v \in V} \lambda_v v | \exists \lambda \in \bigcup\nolimits_{i=1}^d P(T^i) }.
\]
In other words, we can adapt a formulation for $\bigcup_{i=1}^d P(T^i)$ to one for $\bigcup_{i=1}^d S^i$ by introducing auxiliary convex multiplier variables $\lambda \in \Delta^V$. This readily generalizes to the case where the alternatives $S^i$ are unbounded (provided standard representability conditions hold~\cite{Vielma:2015}); see Section~\ref{ssec:discontinuous} for an example.


\subsection{Existing techniques}
There are two standard formulation techniques which are generic in the sense that they can be applied to any combinatorial disjunctive constraint. The first is the well-known big-$M$ formulation approach; see, for example, \cite[Corollary 6.3]{Vielma:2015}. Big-$M$ formulations are simple to reason about and small, though they will, in general, fail to be ideal. A second approach applies the techniques of Jeroslow and Lowe~\cite{Jeroslow:1984} to construct ideal \emph{extended} formulations, requiring additional continuous variables: in particular, a copy of the $\lambda_v$ variables for each alternative $S^i$ and each $v \in S^i$~\cite[Corollary 5.2]{Vielma:2015}. Despite the modest increase in size of the formulation, it has been observed that the resulting formulation can perform worse than expected, potentially due to the block structure of the constraints~\cite{Vielma:2018}.

Alternatively, there exist a number of formulation techniques which are structure-dependent. While not applicable for every combinatorial disjunctive constraint, they can often be deployed to construct very small ideal formulations, the canonical example being the so-called ``logarithmic'' MIP formulations for the special ordered set of type 2 constraint~\cite{Beale:1970}, which have proven computational efficacy~\cite{Vielma:2015,Huchette:2017,Vielma:2010,Vielma:2009a}. One example of a structure-dependent technique is a combinatorial approach previously proposed by the authors which is applicable for instances where a graph associated with the extreme points of the alternatives satisfies a certain condition~\cite{Huchette:2016a}. The size of the resulting formulation is dependent on the size of a decomposition of this graph. Alternatively, there exists a geometric technique that characterizes all possible non-extended ideal formulations for the particular case of the SOS2 constraint~\cite{Huchette:2017,Vielma:2016}. The technique produces a formulation whose size is determined by the number of spanning hyperplanes induced by a particular set of vectors.

This paper builds on this geometric approach by greatly expanding its applicability. In particular, we present a technique to construct ideal MIP formulations, based on hyperplane arrangements, that can be applied to any combinatorial disjunctive constraint satisfying a mild connectivity condition.

\subsection{The embedding approach}

We construct formulations for combinatorial disjunctive constraints through the \emph{embedding approach}~\cite{Vielma:2016}, which works as follows. Assign each alternative $P(T^i)$ a unique vector $h^i \in \bbR^r$. We call such a collection of distinct vectors $H = (h^i)_{i=1}^d$ an \emph{encoding}. Given $\calT$ and $H$, embed the disjunctive constraint in a higher dimensional space as:
\[
    \Em(\calT,H) \defeq \bigcup\nolimits_{i=1}^d (P(T^i) \times \{h^i\}).
\]
If a natural geometric condition is satisfied, then the convex hull of the embedding, $Q(\calT,H) \defeq \Conv(\Em(\calT,H))$, immediately gives an ideal formulation.

\begin{definition}
    A set $H \subset \bbR^r$ is \emph{in convex position} if $\ext(\Conv(H)) = H$, it is \emph{hole-free} if $\Conv(H) \cap \bbZ^r = H$.
\end{definition}

\begin{proposition}
    If $H$ is hole-free and in convex position, then $Q(\calT,H)$ is the LP relaxation for an ideal formulation of $\bigcup_{i=1}^d P(T^i)$.
\end{proposition}
\begin{proof}
Follows as a straightforward extension of \cite[Proposition 1]{Vielma:2016}.
\end{proof}

We present two encodings that we will make repeated use of later.

\begin{definition}
    Take the families of matrices $\left(K^s \in \{0,1\}^{2^s \times s}\right)_{s \in \bbN}$ and $\left(C^s \in \bbZ^{2^s \times s}\right)_{s \in \bbN}$ defined recursively, where $K^1 = C^1 = (0,1)^T$ and for each $s \in \bbN$,
    \[
        K^{s+1} = \begin{pmatrix} K^s & & \mathbf{0}^{2^s} \\ \rev(K^s) & & \mathbf{1}^{2^s} \end{pmatrix} \quad\text{and}\quad
        C^{s+1} = \begin{pmatrix} C^s & & \mathbf{0}^{2^s} \\C^s + \mathbf{1}^{2^s}\otimes C^s_{2^s} & & \mathbf{1}^{2^s} \end{pmatrix},
    \]
    where $\mathbf{0}^t, \mathbf{1}^t \in \{0,1\}^t$ are $t$-dimensional vectors of all zeros and ones, respectively, $\rev(A)$ reverses the rows of the matrix $A$, $u\otimes v=uv^T\in\mathbb{R}^{m\times n}$ for any $u\in \mathbb{R}^m$ and $v\in \mathbb{R}^n$, and $A_{i}$ is the $i$-th row of matrix $A$.

    For $d \in \bbN$ and $r = \lceil \log_2(d) \rceil$, the \emph{binary reflected Gray encoding} $\HRB_d \defeq (K^r_i)_{i=1}^d$ and the \emph{zig-zag encoding} $\HZZ_d \defeq (C^r_i)_{i=1}^d$ are given by the sequence of the first $d$ rows of $K^r$ and $C^r$, respectively.
\end{definition}

As a concrete example,
\[
    K^{3} = \begin{pmatrix} 0 & 0 & 0 \\
                            1 & 0 & 0 \\
                            1 & 1 & 0 \\
                            0 & 1 & 0 \\
                            0 & 1 & 1 \\
                            1 & 1 & 1 \\
                            1 & 0 & 1 \\
                            0 & 0 & 1
             \end{pmatrix} \quad \text{and} \quad
    C^{3} = \begin{pmatrix} 0 & 0 & 0 \\
                            1 & 0 & 0 \\
                            1 & 1 & 0 \\
                            2 & 1 & 0 \\
                            2 & 1 & 1 \\
                            3 & 1 & 1 \\
                            3 & 2 & 1 \\
                            4 & 2 & 1
             \end{pmatrix}.
\]

\section{The main result}

To use an embedding formulation in practice, we need an explicit outer (inequality) description of $Q(\calT,H)$. Our main technical result provides just such a description, hinging on the computation of a particular set of spanning hyperplanes. This result generalizes those of Huchette and Vielma~\cite{Vielma:2016,Huchette:2017}, which only apply to the special case where $\calT = (\{i,i+1\})_{i=1}^d$ (i.e. the SOS2 constraint of Beale and Tomlin~\cite{Beale:1970}).

To simplify our discussion, we will assume without loss of generality (w.l.o.g.) that $V = \llbracket n \rrbracket \defeq \{1,2,\ldots,n\}$.

\begin{theorem} \label{thm:general-cdc-characterization}
    Given the family of sets $\calT = (T^i \subseteq \llbracket n \rrbracket)_{i=1}^d$ and an encoding $H = (h^i)_{i=1}^d \subset \bbR^r$ in convex position, presume that $\llbracket n \rrbracket = \bigcup_{i=1}^d T^i$ and take:
    \begin{itemize}
        \item the intersection directed graph $D = \Set{(i,j) \in \sidx{d}^2 | T^i \cap T^j \neq \emptyset,\ i<j}$,
        \item the set of difference directions $C = \{c^{i,j} \defeq h^j - h^i\}_{(i,j) \in D}$,
        \item the ambient linear space $\calL = \Span(C)$,
        \item the hyperplane defined by $b$ in $\calL$: $M(b;\calL) \defeq \Set{y \in \calL | b \cdot y = 0}$, and
        \item the normal directions $\{b^k\}_{k=1}^\Gamma\subset \bbR^r \backslash \{{\bf 0}^r\}$ for the hyperplanes $\{M(b^k;\calL)\}_{k=1}^\Gamma$ spanned by $C$ in $\calL$.
    \end{itemize}
    If $\dim(\Span(C)) = \dim(H)$, then $(\lambda,z) \in Q(\calT,H)$ if and only if
    \begin{subequations} \label{eqn:general-V-formulation}
    \begin{gather}
        \sum\nolimits_{v=1}^n \min_{s : v \in T^s}\{b^k \cdot h^s\} \lambda_v \leq b^k \cdot z \leq \sum\nolimits_{v=1}^n \max_{s : v \in T^s}\{b^k \cdot h^s\} \lambda_v \quad \forall k \in \llbracket \Gamma \rrbracket \label{eqn:general-V-formulation-1} \\
        (\lambda,z) \in \Delta^{n} \times \aff(H). \label{eqn:general-V-formulation-2}
    \end{gather}
    \end{subequations}
\end{theorem}

We defer the proof of the result to Appendix~\ref{app:proof}. We can present a straightforward sufficient condition that ensures that the dimensionality conditions of Theorem~\ref{thm:general-cdc-characterization} are satisfied. Notationally, take $[r]^2 \defeq \Set{\{i,j\} | i, j \in \llbracket r \rrbracket,\: i \neq j}$.

\begin{proposition} \label{prop:sufficient-condition}
If the directed graph $G=(\llbracket d \rrbracket,D)$ is (weakly) connected, then $\dim(\Span(C)) = \dim(H)$.
\end{proposition}
We defer the proof of the result to Appendix~\ref{app:sufficient-condition}.

\section{Applications of the main result}

\subsection{Univariate piecewise linear functions (continuous or discontinuous)} \label{ssec:discontinuous}
We can apply Theorem~\ref{thm:general-cdc-characterization} to construct small, strong MIP formulations for univariate piecewise linear functions that have a sufficiently long continuous sub-piece. This application includes as special cases the existing best-of-breed MIP formulations for continuous piecewise linear functions (the ``logarithmic'' formulation of Vielma and Nemhauser~\cite{Vielma:2016,Vielma:2009a} and the ``zig-zag'' formulation of Huchette and Vielma~\cite{Huchette:2017}), and improves upon standard formulations described in~\cite{Vielma:2010} for discontinuous piecewise linear functions.

Consider a lower semi-continuous univariate piecewise linear function
\[
    f(x) = \begin{cases}
        a_1 x + b_1 & t_1 \leq x \leq t_2 \\
        \vdots \\
        a_d x + b_d & t_d < x \leq t_{d+1}.
    \end{cases}
\]
We can model its epigraph as a union of polyhedra:
\[
    \epi(f) = \bigcup\nolimits_{i=1}^d \Set{ (x,a_ix+b_i+\gamma) | t_i \leq x \leq t_{i+1},\; \gamma \geq 0 }.
\]
To formulate this set, \cite{Vielma:2010} duplicates the $\lambda$ variables for each interior breakpoint $\{t_i\}_{i=2}^d$. If $r(i)=\lfloor 1+ i/2\rfloor$ and $s(i)=\lceil i/2\rceil$, this takes the form
\begin{equation}\label{olddiscont}
\epi(f)=\sum\nolimits_{i=1}^{2d}\lambda_i\left(t_{r(i)},a_{s(i)}t_{r(i)}+b_{s(i)}+\gamma\right),\; \gamma \geq 0, \; \lambda\in \Delta^{2d},
\text{ $\lambda$ is SOS2},
\end{equation}
where we have constrained $\lambda$ to satisfy the standard SOS2 constraint on $2d$ breakpoints. However, this approach is inefficient as it unnecessarily duplicates breakpoints where $f$ is continuous (i.e. those $i$ where $a_it_{i+1} + b_i = a_{i+1}t_{i+1} + b_{i+1}$). We can remove this redundancy by considering the family of sets
\[
    T^i \defeq \{(t_i,a_it_i+b_i), (t_{i+1},a_it_{i+1}+b_i)\} \quad \forall i \in \llbracket d \rrbracket.
\]
We can then identify the ground set $\bigcup_{i=1}^d T^i$ with $V = \llbracket d+1+\kappa \rrbracket$, where $\kappa$ is the number of discontinuous breakpoints, i.e. those $i \in \llbracket d \rrbracket$ where $a_it_{i+1} + b_i \neq a_{i+1}t_{i+1} + b_{i+1}$. \blue{The mapping between $\bigcup_{i=1}^d T^i$ and $V$ is the natural extension of that in \eqref{olddiscont} where all breakpoints are duplicated.} Provided that $f$ has a sufficiently long continuous sub-piece, we can construct a logarithmically-sized ideal MIP formulation for $\epi(f)$.

\begin{proposition} \label{prop:pwl}
Take $d = 2^r$ for some integer $r \geq 2$, and assume that $f$ is continuous on the interval $[t_{d/4+1},t_{d/2+1}]$, inclusive of the endpoints. Select an encoding $H=(h^i)_{i=1}^d$ that is either $\HRB_d$ or $\HZZ_d$, and take $h^0 \equiv h^1$ and $h^{d+1} \equiv h^d$ for notational simplicity. Then $(\lambda,z) \in Q(\calT,H)$ if and only if
    \begin{subequations} \label{eqn:discontinuous-pwl}
    \begin{gather}
        \sum\nolimits_{v=1}^{d+1+\kappa} \lambda_v\min_{s : v \in T^s}h^{s}_k \leq z_k \leq \sum\nolimits_{v=1}^{d+1+\kappa} \lambda_v\max_{s : v \in T^s}h^s_k \quad \forall k \in \llbracket r \rrbracket \\
        (\lambda,z) \in \Delta^{d+1+\kappa} \times \bbR^r.
    \end{gather}
    \end{subequations}
    \blue{Moreover, a directly analogous result can be obtained with continuity on the interval $[t_{d/2+1},t_{3d/4+1}]$.}
\end{proposition}
\begin{proof}
The result follows by inspection of the recursive definitions of $\HRB_d$ and $\HZZ_d$, seeing that under the partial continuity assumption that $C \supseteq \{{\bf e}^k\}_{k=1}^r$, and so $C$ has dimension $r$. Furthermore, as $C \subseteq \{\pm {\bf e}^k\}_{k=1}^r$, we can apply Theorem~\ref{thm:general-cdc-characterization} using the normal directions $\{b^k = {\bf e}^k\}_{k=1}^r$, giving the result.
\end{proof}
Applying one of the existing logarithmic formulations for the SOS2 constraint in \eqref{olddiscont} yields a formulation with the same number of binary variables and general inequalities as \eqref{eqn:discontinuous-pwl}. However, the formulation based on \eqref{olddiscont}  will have $2d$ continuous variables even if the number of discontinuous breakpoints $\kappa$ is small or even zero In contrast, we observe that formulation \eqref{eqn:discontinuous-pwl} uses only $d+1+\kappa$ continuous variables. A similar favorable accounting occurs when comparing \eqref{eqn:discontinuous-pwl} with the ``disaggregated logarithmic'' formulation described in \cite{Vielma:2010}. Additionally, when $\kappa=0$, \eqref{eqn:discontinuous-pwl} is equivalent to the logarithmic formulations for continuous functions from \cite{Huchette:2017,Vielma:2009a,Vielma:2016}.

\subsection{Relaxations of the annulus}

The annulus is a set in the plane $\calA = \Set{x \in \bbR^2 | L \leq ||x||_2 \leq U}$ for constants $L,U \in \bbR_{\geq 0}$; see the left side of Figure~\ref{fig:annulus} for an illustration. A constraint of the form $x \in \calA$ might arise when modeling a complex number $z = x_1 + x_2 \mathbf{i}$, as $x \in \calA$ bounds the magnitude of $z$ as $L \leq |z| \leq U$. Such constraints arise in power systems optimization: for example, in the ``rectangular formulation''~\cite{Kocuk:2015} and the second-order cone reformulation~\cite{Jabr:2006,Liu:2017} of the optimal power flow problem and its voltage stability-constrained variant~\cite{Cui:2017}, and the reactive power dispatch problem~\cite{Foster:2013}. Another application is footstep planning in robotics~\cite{Deits:2014,Kuindersma:2016}, where $L=U=1$, $x = (\cos(\theta),\sin(\theta))$, and $x$ must satisfy the trigonometric identity $x_1^2 + x_2^2 = 1$.

When $0 < L \leq U$, $\calA$ is a nonconvex set. Moreover, the annulus is not \emph{mixed-integer convex representable}~\cite{Lubin:2017,Lubin:2017a}: that is, there do not exist mixed-integer formulations for the annulus even if we allow the relaxation $R$ to be an arbitrary convex set. Foster~\cite{Foster:2013} proposes a disjunctive relaxation for the annulus given as $\hat{\calA} \defeq \bigcup_{i=1}^d S^i$, where each
\begin{equation} \label{eqn:annulus-pieces}
    S^i = \Conv\left(\left\{v^{2i+s-4}\right\}_{s=1}^4\right) \quad \forall i \in \llbracket d \rrbracket
\end{equation}
is a quadrilateral whose extreme points are
\begin{alignat*}{3}
    v^{2i-1} &= \left(L\cos\left(\frac{2\pi i}{d}\right), L\sin\left(\frac{2\pi i}{d}\right)\right) \quad\quad &\forall i \in \llbracket d \rrbracket& \\
    v^{2i} &= \left(U\sec\left(\frac{2\pi }{d}\right)\cos\left(\frac{2\pi i}{d}\right), U\sec\left(\frac{2\pi }{d}\right)\sin\left(\frac{2\pi i}{d}\right)\right) \quad\quad &\forall i \in \llbracket d \rrbracket&,
\end{alignat*}
taking $v^{0} \equiv v^{2d}$ and $v^{-1} \equiv v^{2d-1}$ for notational simplicity. We can in turn represent this disjunctive relaxation via the family of sets $\calT = (T^i = \{2i+s-4\}_{s=1}^4)_{i=1}^d$; see the right side of Figure~\ref{fig:annulus} for an illustration.

We start by applying Theorem~\ref{thm:general-cdc-characterization} to construct an ideal MIP formulation with $\lceil \log_2(d) \rceil$ integer variables and $2\lceil \log_2(d) \rceil$ inequality constraints.

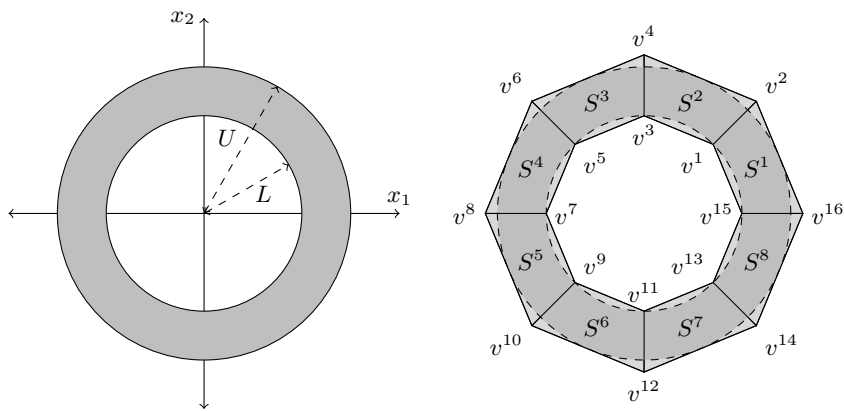
\begin{figure}[htpb]
    \centering
    \begin{tikzpicture}[scale=.65]

        \draw [<->] (-4,0) -- (4,0);
        \draw [<->] (0,-4) -- (0,4);
        \node [left] at (0,4) {$x_2$};
        \node [above] at (4,0) {$x_1$};
        \draw[fill=gray!50,even odd rule]  circle (3) circle (2);

        \draw [<->,dashed] (0,0) -- (1.7320508075688774,1);
        \draw [<->,dashed] (0,0) -- (1.5,2.5980762113533156);

        \node [right] at (1.7320508075688774/2,1/2-0.1) {$L$};
        \node [above left] at (1.5/2+0.1,2.5980762113533156/2-0.1) {$U$};

    \begin{scope}[shift={(9,0)}]
        \draw [fill=gray!30] (1.4142135623730951,1.414213562373095) -- (2.0,0.0) -- (3.2471766008771823,0.0) -- (2.296100594190539,2.2961005941905386);
        \draw [fill=gray!30] (1.2246467991473532e-16,2.0) -- (1.4142135623730951,1.414213562373095) -- (2.296100594190539,2.2961005941905386) -- (1.988322215265212e-16,3.2471766008771823);
        \draw [fill=gray!30] (-1.414213562373095,1.4142135623730951) -- (1.2246467991473532e-16,2.0) -- (1.988322215265212e-16,3.2471766008771823) -- (-2.2961005941905386,2.296100594190539);
        \draw [fill=gray!30] (-2.0,2.4492935982947064e-16) -- (-1.414213562373095,1.4142135623730951) -- (-2.2961005941905386,2.296100594190539) -- (-3.2471766008771823,3.976644430530424e-16);
        \draw [fill=gray!30] (-1.4142135623730954,-1.414213562373095) -- (-2.0,2.4492935982947064e-16) -- (-3.2471766008771823,3.976644430530424e-16) -- (-2.2961005941905395,-2.2961005941905386);
        \draw [fill=gray!30] (-3.6739403974420594e-16,-2.0) -- (-1.4142135623730954,-1.414213562373095) -- (-2.2961005941905395,-2.2961005941905386) -- (-5.964966645795635e-16,-3.2471766008771823);
        \draw [fill=gray!30] (1.4142135623730947,-1.4142135623730954) -- (-3.6739403974420594e-16,-2.0) -- (-5.964966645795635e-16,-3.2471766008771823) -- (2.296100594190538,-2.2961005941905395);
        \draw [fill=gray!30] (2.0,-4.898587196589413e-16) -- (1.4142135623730947,-1.4142135623730954) -- (2.296100594190538,-2.2961005941905395) -- (3.2471766008771823,-7.953288861060848e-16);
        \draw[fill=gray!50,even odd rule,dashed]  circle (3) circle (2);

        \node at (2.309698831278217, 0.9567085809127245) {$S^1$};
        \node at (0.9567085809127246, 2.309698831278217) {$S^2$};
        \node at (-0.9567085809127243, 2.309698831278217) {$S^3$};
        \node at (-2.309698831278217, 0.9567085809127247) {$S^4$};
        \node at (-2.309698831278217, -0.9567085809127241) {$S^5$};
        \node at (-0.9567085809127258, -2.309698831278216) {$S^6$};
        \node at (0.956708580912725, -2.3096988312782165) {$S^7$};
        \node at (2.309698831278216, -0.956708580912726) {$S^8$};

        \draw (1.4142135623730951,1.414213562373095) -- (2.0,0.0) -- (3.2471766008771823,0.0) -- (2.296100594190539,2.2961005941905386);
        \draw (1.2246467991473532e-16,2.0) -- (1.4142135623730951,1.414213562373095) -- (2.296100594190539,2.2961005941905386) -- (1.988322215265212e-16,3.2471766008771823);
        \draw (-1.414213562373095,1.4142135623730951) -- (1.2246467991473532e-16,2.0) -- (1.988322215265212e-16,3.2471766008771823) -- (-2.2961005941905386,2.296100594190539);
        \draw (-2.0,2.4492935982947064e-16) -- (-1.414213562373095,1.4142135623730951) -- (-2.2961005941905386,2.296100594190539) -- (-3.2471766008771823,3.976644430530424e-16);
        \draw (-1.4142135623730954,-1.414213562373095) -- (-2.0,2.4492935982947064e-16) -- (-3.2471766008771823,3.976644430530424e-16) -- (-2.2961005941905395,-2.2961005941905386);
        \draw (-3.6739403974420594e-16,-2.0) -- (-1.4142135623730954,-1.414213562373095) -- (-2.2961005941905395,-2.2961005941905386) -- (-5.964966645795635e-16,-3.2471766008771823);
        \draw (1.4142135623730947,-1.4142135623730954) -- (-3.6739403974420594e-16,-2.0) -- (-5.964966645795635e-16,-3.2471766008771823) -- (2.296100594190538,-2.2961005941905395);
        \draw (2.0,-4.898587196589413e-16) -- (1.4142135623730947,-1.4142135623730954) -- (2.296100594190538,-2.2961005941905395) -- (3.2471766008771823,-7.953288861060848e-16);

        \node [left] at (2.0,0.0) {$v^{15}$};
        \node [right] at (3.2471766008771823,0.0) {$v^{16}$};
        \node [below left] at (1.4142135623730951,1.414213562373095) {$v^{1}$};
        \node [above right] at (2.296100594190539,2.2961005941905386) {$v^{2}$};
        \node [below] at (1.2246467991473532e-16,2.0) {$v^{3}$};
        \node [above] at (1.988322215265212e-16,3.2471766008771823) {$v^{4}$};
        \node [below right] at (-1.414213562373095,1.4142135623730951) {$v^{5}$};
        \node [above left] at (-2.2961005941905386,2.296100594190539) {$v^{6}$};
        \node [right] at (-2,0) {$v^{7}$};
        \node [left] at (-3.2471766008771823,0.0) {$v^{8}$};
        \node [above right] at (-1.414213562373095,-1.4142135623730951) {$v^{9}$};
        \node [below left] at (-2.2961005941905386,-2.296100594190539) {$v^{10}$};
        \node [above] at (-1.2246467991473532e-16,-2.0) {$v^{11}$};
        \node [below] at (-1.988322215265212e-16,-3.2471766008771823) {$v^{12}$};
        \node [above left] at (1.4142135623730951,-1.414213562373095) {$v^{13}$};
        \node [below right] at (2.296100594190539,-2.2961005941905386) {$v^{14}$};

    \end{scope}
    \end{tikzpicture}
    \caption{(\textbf{Left}) The annulus $\calA$ and (\textbf{Right}) its corresponding quadrilateral relaxation $\hat{\calA}$ given by \eqref{eqn:annulus-pieces} with $d=8$.}
    \label{fig:annulus}
\end{figure}

\begin{proposition}\label{prop:log-annulus}
    Fix $d = 2^r$ for some $r \in \bbN$. Take the binary reflected Gray encoding $\HRB_d = (h^i)_{i=1}^d$, along with $h^0 \equiv h^d$ for notational convenience. Then $(\lambda,z) \in Q(\calT,\HRB_d)$ if and only if
    \begin{subequations} \label{eqn:annulus-log-form}
    \begin{gather}
        \sum_{i=1}^d \min\{h^{i-1}_k,h^i_k\} (\lambda_{2i-1} + \lambda_{2i}) \leq z_k \leq \sum_{i=1}^d \max\{h^{i-1}_k,h^i_k\} (\lambda_{2i-1} + \lambda_{2i}) \quad \forall k \in \llbracket r \rrbracket \\
        (\lambda,z) \in \Delta^{2d} \times \bbR^r.
    \end{gather}
    \end{subequations}
\end{proposition}
\proof{}
    The result follows from Theorem~\ref{thm:general-cdc-characterization} after observing that $D = \{(i,i+1)\}_{i=1}^{d-1} \cup (1,d)$ and therefore that $C = \{\pm{\bf e}^k\}_{k=1}^r$, as the binary reflected Gray encoding is cyclic (i.e. $h^{d}-h^1 = {\bf e}^1$).
\endproof

We can also apply Theorem~\ref{thm:general-cdc-characterization} using the zig-zag encoding to produce a MIP formulation for the annulus with $\lceil \log_2(d) \rceil$ integer variables and $\scrO(\log^2(d))$ general inequality constraints. Despite the modest increase in the number of constraints, the analysis of Huchette and Vielma~\cite{Huchette:2017} shows that the the zig-zag encoding enjoys the ``incremental branching'' behavior for univariate piecewise linear functions, leading to improved computational performance relative to the logarithmic formulation of Vielma et al.~\cite{Vielma:2010,Vielma:2009a}. Therefore, it may be the case that the zig-zag formulation for the annulus similarly outperforms the logarithmic formulation \eqref{eqn:annulus-log-form}, despite the modest increase in the number of constraints.

\begin{proposition} \label{prop:zig-zag-annulus}
    Fix $d = 2^r$ for some $r \in \bbN$. Take the zig-zag encoding $\HZZ_d = (h^i)_{i=1}^d$, along with $h^0 \equiv h^d$ for notational convenience. Then $(\lambda,z) \in Q(\calT,\HZZ_d)$ if and only if
    \begin{align*}
        \sum_{i=1}^d \min\{h^{i-1}_k,h^i_k\} (\lambda_{2i-1} + \lambda_{2i}) &\leq z_k \quad \forall k \in \llbracket r \rrbracket \\
        \sum_{i=1}^d \max\{h^{i-1}_k,h^i_k\} (\lambda_{2i-1} + \lambda_{2i}) &\geq z_k \quad \forall k \in \llbracket r \rrbracket \\
        \sum_{i=1}^d \min\left\{\frac{h^{i-1}_k}{2^{\ell}}-\frac{h^{i-1}_\ell}{2^{k}},\frac{h^{i}_k}{2^{\ell}}-\frac{h^{i}_\ell}{2^{k}}\right\} (\lambda_{2i-1} + \lambda_{2i}) &\leq \frac{z_k}{2^\ell} - \frac{z_\ell}{2^k} \quad \forall \{k,\ell\} \in [r]^2 \\
        \sum_{i=1}^d \max\left\{\frac{h^{i-1}_k}{2^{\ell}}-\frac{h^{i-1}_\ell}{2^{k}},\frac{h^{i}_k}{2^{\ell}}-\frac{h^{i}_\ell}{2^{k}}\right\} (\lambda_{2i-1} + \lambda_{2i}) &\geq \frac{z_k}{2^\ell} - \frac{z_\ell}{2^k} \quad \forall \{k,\ell\} \in [r]^2 \\
        (\lambda,z) &\in \Delta^{2d} \times \bbR^r.
    \end{align*}
\end{proposition}
\proof{}
    The result follows from Theorem~\ref{thm:general-cdc-characterization}. As $D = \{(i,i+1)\}_{i=1}^{d-1} \cup (1,d)$, it follows that $C = \{{\bf e}^k\}_{k=1}^r \cup \{w \equiv (2^{r-1},2^{r-2},\ldots,2^0)\}$. We have that $B = \{{\bf e}^k\}_{k=1}^r$ induce all hyperplanes spanned by the vectors $C \backslash \{w\} = \{{\bf e}^k\}_{k=1}^r$. Now consider each hyperplane spanned by $\hat{C} = \{{\bf e}^k\}_{k \in I} \cup \{w\} \subset C$, where $I \subseteq \llbracket r \rrbracket$. As $|C| = r+1$ and $\dim(\Span(C)) = r$, we must have $|I| = r-2$, i.e. that there are distinct indices $k,\ell \in \llbracket r \rrbracket \backslash I$ where $I \cup \{k,\ell\} = \llbracket r \rrbracket$. We may then compute that the corresponding hyperplane is given by the normal direction $b^{k,\ell} \defeq 2^{-\ell}{\bf e}^{k} - 2^{-k}{\bf e}^{\ell}$. Therefore, we have that the set $B = \{{\bf e}^k\}_{k=1}^r \cup \{b^{k,\ell}\}_{\{k,\ell\} \in [r]^2}$ suffices for the conditions of Theorem~\ref{thm:general-cdc-characterization}, giving the result.
\endproof

\begin{appendix}

\section{Proof of Theorem~\ref{thm:general-cdc-characterization}} \label{app:proof}
\proof{}
    We start by taking $B = \ext(Q(\calT,H)) = \Set{({\bf e}^w,h^j) | w \in T^j}$ as the set of all extreme points of $Q(\calT,H)$. It is straightforward to show that each $(\lambda,z) \in Q(\calT,H)$ satisfies \eqref{eqn:general-V-formulation}. Take some $(\hat{\lambda},\hat{z}) = ({\bf e}^w,h^j) \in B$. Clearly \eqref{eqn:general-V-formulation-2} is satisfied, and we have that for each $k \in \llbracket \Gamma \rrbracket$,
    \[
        \sum\nolimits_{v=1}^n \min_{s : v \in T^s}\{b^k \cdot h^s\} \hat{\lambda}_v = \min_{s : w \in T^s}\{b^k \cdot h^s\} \leq b^k \cdot h^j = b^k \cdot \hat{z},
    \]
    where the inequality follows as $w \in T^j$. An identical argument follows for the other inequality in \eqref{eqn:general-V-formulation-1}. This implies that each $(\lambda,z) \in Q(\calT,H)$ satisfies \eqref{eqn:general-V-formulation}, giving one direction of the characterization.

    For the other direction, let $F$ be a facet of $Q(\calT,H)$. By possibly adding or subtracting multiples of $\sum_{i=1}^n \lambda_i = 1$ (an equation appearing in the definition of $\Delta^n$) and the equations defining $\aff(H)$, we may assume w.l.o.g. that $F$ is induced by $\tilde{a} \cdot \lambda \leq \tilde{b} \cdot y$ for some $(\tilde{a},\tilde{b}) \in \bbR^{n+r}$.  As $F$ is a facet, it is supported by some strict nonempty subset of extreme points $\tilde{B} \subsetneq B$. Take $\tilde{D} = \Set{(i,j) \in D | \exists v \in \llbracket n \rrbracket \text{ s.t. } ({\bf e}^v,h^i),({\bf e}^v,h^j) \in \tilde{B}}$ and $\tilde{C} = \Set{c^{i,j} \in C | (i,j) \in \tilde{D}}$. In particular, we see that $\tilde{b} \cdot c^{i,j} = 0$ for each $c^{i,j} \in \tilde{C}$, as if $(i,j) \in \tilde{D}$, this implies that there is some $v \in \llbracket n \rrbracket$ whereby $\tilde{a} \cdot {\bf e}^v = \tilde{b} \cdot h^i = \tilde{b} \cdot h^j$.

    As $F$ is a proper face, we know that there is some point in $B$ not supporting $F$, w.l.o.g. $({\bf e}^1,h^1) \in B \backslash \tilde{B}$. We will take all the remaining extreme points not in $\tilde{B}$ as $B^\star = B \backslash (\tilde{B} \cup \{({\bf e}^1,h^1)\})$. Consider three cases for the dimension of $\Span(\tilde{C})$ which exhaust all possibilities.

    \paragraph{\underline{Case 1: $\dim(\Span(\tilde{C})) = \dim(\Span(C))$}}
    In this case, we show that $F$ corresponds to a variable bound on a single component of $\lambda$. As $\tilde{C} \subseteq C$ and $\dim(\Span(\tilde{C})) = \dim(\Span(C))$, we conclude that $\Span(\tilde{C}) = \Span(C) \equiv \calL$. Then $\tilde{b} \in \calL^\perp$, as $\tilde{b} \perp \tilde{C}$. Furthermore, $\calL$ is the linear space parallel to $\aff(H)$. Therefore, we can w.l.o.g. presume that $\tilde{b} = {\bf 0}^r$, as \eqref{eqn:general-V-formulation-2} constrains $z \in \aff(H)$.

    We observe that $\tilde{a} \neq {\bf 0}^n$, as otherwise this would correspond to the vacuous inequality $0 \leq 0$, which is not a proper face. We now show that $\tilde{a}$ has exactly one nonzero element. Assume for contradiction that this is not the case, and $\tilde{a}_p, \tilde{a}_q < 0$ for some distinct $p,q\in\sidx{n}$ (any strictly positive components will not yield a valid inequality for $B$). This implies that there exists some $j, k \in \llbracket d \rrbracket$ such that at least one of $({\bf e}^p,h^{j})$ and $({\bf e}^q,h^{k})$ is not contained in $B^\star$; assume w.l.o.g. that $({\bf e}^p,h^{j}) \not\in B^\star$. However, we could then perform the simple tilting $\tilde{a}_q \leftarrow 0$ to construct a distinct face with strictly larger support, as now $({\bf e}^q,h^j)$ is supported by the corresponding face for each $j$ such that $q \in T^j$. Furthermore, as this new constraint does not support $({\bf e}^p,h^j)$ for each $j$ such that $p \in T^j$, the new face is proper, and thus contradicts the original face $F$ being a facet. Therefore, we can normalize the coefficients to $\tilde{a} = -{\bf e}^p$, giving a variable bound constraint on a component of $\lambda$ which appears in the restriction $\lambda \in \Delta^n$ in \eqref{eqn:general-V-formulation-2}.

    \paragraph{\underline{Case 2: $\dim(\Span(\tilde{C})) = \dim(\Span(C)) - 1$}}
    The fact that $\tilde{b} \perp \tilde{C}$, along with the dimensionality of $\tilde{C}$, implies that $M(\tilde{b};\calL) = \Span(\tilde{C})$ is a hyperplane in $\calL$. This means we can assume w.l.o.g. that $\tilde{b} = s b^k$ for some $k \in \llbracket \Gamma \rrbracket$ and $s \in \{-1,+1\}$. We then compute for each $v \in \llbracket n \rrbracket$ that either $a_v = \min_{j : v \in T^j}\{b^k \cdot h^j\}$ if $s=+1$, or $a_v = -\max_{j : v \in T^j}\{b^k \cdot h^j\}$ if $s = -1$. Therefore, this facet is included in \eqref{eqn:general-V-formulation-1}.
    \begin{figure}[htpb]
        \centering
        \begin{tikzpicture}[scale=0.70]
            \draw [fill=gray!50,draw=none] (-1,1) -- (1,-1) -- (5,-1) -- (5,4) -- (-1,4) -- (-1,1);
            \draw [fill=gray!80,draw=none] (0,0) -- (1,4) -- (5,4) -- (5,1) -- (0,0);
            \draw [->] (0,0) -- (1,4);
            \draw [->] (0,0) -- (5,1);
            \draw [line width=1.5] (-1,-1) -- (4,4);
            \draw [fill=black] (2,2) circle [radius=0.1];
            \node [right] at (2,2) {\footnotesize $(a,b)$};
            \node [above left] at (5,1) {\footnotesize $K^2$};
            \node [above left] at (0,-1) {\footnotesize $L$};
            \node [above right] at (1,-1) {\footnotesize $\hat{K}^2$};
            \draw [line width=1.1, dashed] (-1,1) -- (1,-1);
        \end{tikzpicture}\hfill
        \begin{tikzpicture}[scale=0.70]
            \draw [fill=gray!20,draw=none] (-1,-1) -- (-1,4) -- (5,4) -- (5,-1) -- (-1,-1);
            \draw [fill=gray!50,draw=none] (-1,1) -- (1,-1) -- (5,-1) -- (5,4) -- (-1,4) -- (-1,1);
            \draw [fill=gray!80,draw=none] (0,0) -- (1,4) -- (5,4) -- (5,1) -- (0,0);
            \draw [->] (0,0) -- (1,4);
            \draw [->] (0,0) -- (5,1);
            \draw [fill=black] (2,2) circle [radius=0.1];
            \node [right] at (2,2) {\footnotesize $(a,b)$};
            \node [above left] at (5,1) {\footnotesize $K^2$};
            \node [above left] at (0,-1) {\footnotesize $L$};
            \node [above right] at (1,-1) {\footnotesize $\hat{K}^2$};
            \node [left] at (0.75,3) {\footnotesize
            $(\tilde{a},\tilde{b})$};
            \draw [fill=black] (0.75,3) circle [radius=0.1];
            \draw [->, dotted,line width=1.25] (2,2) -- (0.75,3);
            \draw [line width=1.1, dashed] (-1,1) -- (1,-1);
        \end{tikzpicture}
        \caption{\textbf{(Left)} $\dim(L)=1$, and so we cannot ``tilt'' the inequality to increase its support. \textbf{(Right)} $\dim(\Span(C)) = \dim(H)$ implies $\dim(L) > 1$ (i.e. all of $\bbR^2$ in this projected space), and so we can tilt.}
        \label{fig:proof}
    \end{figure}
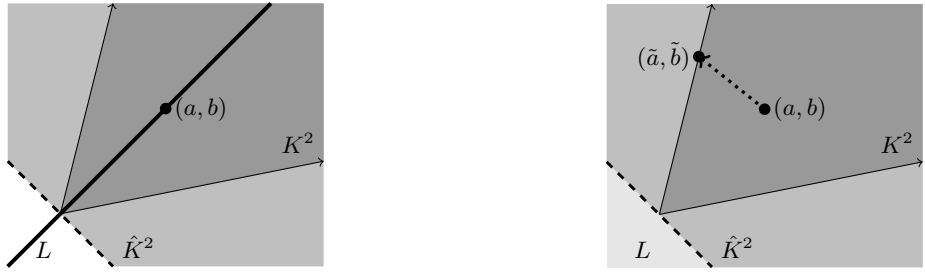
    \paragraph{\underline{Case 3: $\dim(\Span(\tilde{C})) < \dim(\Span(C))-1$}}
    We will show that this case cannot occur if $F$ is a general inequality facet. In fact, observe that if ${\bf e}^i \not\in \Proj_\lambda(\tilde{B})$, then $\tilde{a} \cdot \lambda \leq \tilde{b} \cdot z$ is either equivalent to, or dominated by, the variable bound $\lambda_i \geq 0$. Therefore, we assume that $\Proj_\lambda(\tilde{B}) = \{{\bf e}^i\}_{i=1}^n$ for the remainder.

    Presume for contradiction that it is indeed the case that $F$ is a facet and $\dim(\Span(\tilde{C})) < \dim(\Span(C))-1$. 
    First, we show that $B^\star \neq \emptyset$. If this were not the case, then $\tilde{B} = B \backslash \{({\bf e}^1,h^1)\}$, and hence for each $(i,j) \in D \backslash \tilde{D}$, it must be that $i=1$ and $T^1 \cap T^j = \{1\}$. Then there exists $J\subseteq \{2,\ldots,d\}$ such that $\tilde{C} = C\setminus \Set{c^{1,j} | j\in J}$ and $C_J \defeq \Set{c^{i,j} | i,j\in J, i<j} \subseteq C$. Without loss of generality we may assume that $J=\set{2,\ldots,k}$. Then by noting that $c^{1,j} \equiv h^j-h^1=\sum_{i=2}^{j} c^{i-1,i}=\sum_{i=2}^{j} h^{i}-h^{i-1}$, we conclude that $\dim\left(C_J\setminus \Set{c^{1,j} | j\in J}\right) \geq \dim(C_J) - 1$, which contradicts $\dim(\Span(\tilde{C})) < \dim(\Span(C))-1$. Therefore, we conclude that $B^\star \neq \emptyset$.

    We now define the cone $K = \Set{(a,b) \in \bbR^n \times \calL | a \cdot {\bf e}^v \leq b \cdot h^j \: \forall ({\bf e}^v,h^j) \in B^\star }$
    and the linear space $L = \Set{ (a,b) \in \bbR^n \times \calL | a \cdot {\bf e}^v = b \cdot h^j \: \forall ({\bf e}^v,h^j) \in \tilde{B} }$.

    By definition of $ B^\star$, $(\tilde{a},\tilde{b})\in \relint(K)$ (taken relative to $\bbR^n \times \calL$).
    Next, we show that $\dim(L) > 1$. To show this, we start by instead studying $L' = \Set{ b \in \calL | b \cdot c = 0 \: \forall c \in \tilde{C} }$. We can readily observe that $L' = \Proj_b(L)$. Furthermore, as $\Proj_a(\tilde{B}) = \{{\bf e}^i\}_{i=1}^n$ from the argument at the beginning of the case, we conclude that the set $\Set{a | (a,b) \in L}$ is a singleton. In other words, the values for $a$ are completely determined by the values for $b$ in $L$. From this, we conclude that $\dim(L) = \dim(L')$. From the definition of $L'$, we see that $L'$ and $\Span(\tilde{C})$ form an orthogonal decomposition of $\calL$. Therefore, $\dim(\calL) = \dim(L') + \dim(\Span(\tilde{C}))$. Recalling that $\dim(\calL) = \dim(\Span(C))$, and that we are assuming that $\dim(\Span(\tilde{C})) < \dim(\Span(C))-1$, we have that $\dim(L) = \dim(L') = \dim(\calL) - \dim(\Span(\tilde{C})) = \dim(\Span(C)) - \dim(\Span(\tilde{C})) > 1$, giving the result.

    We now show that $K \cap L$ is pointed. To see this, presume for contradiction that there exists a nonzero $(\hat{a},\hat{b})$ such that $(\hat{a},\hat{b}),(-\hat{a},-\hat{b}) \in K\cap L$. However, this would imply that $\hat{a} \cdot {\bf e}^v = \hat{b} \cdot h^j$ for all $({\bf e}^v,h^j) \in \tilde{B} \cup B^\star$. Because $B^\star\neq\emptyset$, this implies that $\hat{a} \cdot \lambda \leq \hat{b} \cdot z$ is a face strictly containing the facet $F$, and so must be a non-proper face (i.e. it is additionally supported by $({\bf e}^1,h^1)$ and hence by every point in $B$). However, this would imply that $\hat{b} \cdot c = 0$ for all $c \in C$, and as $\calL = \Span(C)$, this would necessitate that $\hat{b} \in \calL^\perp$. As $\hat{b} \in \calL$ from the definition of $K$, it follows that $\hat{b} = {\bf 0}^r$. However, this would imply that $\hat{a} \cdot \lambda = 0$ is valid for $B$, which cannot be the case unless $\hat{a} = {\bf 0}^n$, a contradiction. Therefore, $K \cap L$ is pointed.

    As $\dim(L) > 1$, we can take some two-dimensional linear subspace $L^2 \subseteq L$ such that $(\tilde{a},\tilde{b}) \in L^2$. As $(\tilde{a},\tilde{b}) \in L \cap \relint(K)$, it follows that $(\tilde{a},\tilde{b}) \in L^2 \cap \relint(K)$ as well. Similarly, as $K \cap L$ is pointed, it follows that $K^2 \defeq L^2 \cap K$ is pointed as well. Furthermore, as $K$ is full-dimensional in $\bbR^n \times \calL$, $K^2$ is full-dimensional in $L^2 \subset \bbR^n \times \calL$ (i.e. 2-dimensional).Therefore, a minimal description for it includes the equalities that define $L^2$, along with exactly two nonempty-face-inducing inequality constraints from the definition of $K$. See Figure~\ref{fig:proof} for an illustration of the following argument in this space.

    Add the single strict inequality $\hat{K}^2 \defeq K^2 \cap \Set{(a,b) \in \bbR^n \times \calL | a \cdot {\bf e}^1 < b \cdot h^1}$. As $\tilde{a} \cdot {\bf e}^1 < \tilde{b} \cdot h^1$ and $(\tilde{a},\tilde{b}) \in K^2$, it follows that $\hat{K}^2$ is nonempty and also 2-dimensional, and can be described using only the linear equations defining $L^2$, the strict inequality $a \cdot {\bf e}^1 < b \cdot h^1$, and at least one (and potentially two) of the inequalities previously used to describe $K^2$. Select one of the defining nonempty-face-inducing inequalities given by $a \cdot {\bf e}^v \leq b \cdot h^j$, where $({\bf e}^v,h^j) \in B^\star$.

    Now construct the restriction $S \defeq \Set{(a,b) \in \hat{K}^2 | a \cdot {\bf e}^v = b \cdot h^j}$. As $a \cdot {\bf e}^v \leq b \cdot h^j$ induces a non-empty face on the cone $\hat{K}^2$, $S$ is nonempty. Furthermore, we see that any $(\hat{a},\hat{b}) \in S$ will correspond to a valid inequality $\hat{a} \cdot \lambda \leq \hat{b} \cdot z$ for $B$ with strictly greater support than our original face $\tilde{a} \cdot \lambda \leq \tilde{b} \cdot z$. In particular, we see that $({\bf e}^v,h^j) \in B^\star$, i.e. $\tilde{a} \cdot {\bf e}^v < \tilde{b} \cdot h^j$, but by construction $\hat{a} \cdot {\bf e}^v = \hat{b} \cdot h^j$. Additionally, since $\hat{a} \cdot {\bf e}^1 < \hat{b} \cdot h^1$, the corresponding face is proper, which implies that $F$ cannot be a facet, a contradiction.
\endproof

\section{Proof of Proposition~\ref{prop:sufficient-condition}} \label{app:sufficient-condition}
\proof{}
The result follows by showing that $\calL = \aff(H) - h^1$. For simplicity, take $\breve{D} = \Set{ (i,j) \in [d]^2 | T^i \cap T^j \neq \emptyset}$ as the version of $D$ with all parallel edges added.
To show that $\Span(C) \subseteq \aff(H) - h^1$, take some $z \in \Span(C)$, and so there exist multipliers $\gamma_{i,j}$ such that $z = \sum_{(i,j) \in \breve{D}} \gamma_{i,j} (h^i-h^j) = \sum_{i=1}^d \alpha_i h^i$, where $\alpha_i = \sum_{j : (i,j) \in \breve{D}}\gamma_{i,j} - \sum_{k : (k,i) \in \breve{D}}\gamma_{k,i}$. Then $z = (\alpha_1+1) h^1 + (\sum_{i=2}^d \alpha_i h^i) - h^1$, i.e. $z \in \aff(H) - h^1$, as $(\alpha_1+1) + \sum_{i=2}^d \alpha_i = 1 + \sum_{i=1}^d \left( \sum_{j : (i,j) \in \breve{D}} \gamma_{i,j} - \sum_{k : (k,i) \in \breve{D}} \gamma_{k,i}\right) = 1 + 0$, and so they form affine multipliers.

To show that $\Span(C) \supseteq \aff(H) - h^1$, take some $z \in \aff(H)-h^1$, and so there exists multipliers $\mu_{i}$ such that $z = (\sum_{i=1}^d \mu_i h^i) - h^1$ and $\sum_{i=1}^d \mu_i = 1$. As $G$ is connected, there exists some closed directed path $(t_1 \equiv 1, t_2, \ldots, t_{r},t_{r+1}\equiv 1)$ on $G$ that traverses each node. Take $\alpha_i \defeq -\frac{\mu_i}{\# \text{ of times path traverses } i}$ for each $i \in \llbracket d \rrbracket$. Then $z = \sum_{k=1}^r (h^{t_{k+1}}-h^{t_{k}}) \sum_{\ell=1}^k \alpha_{t_\ell} = \sum_{k=1}^r c^{t_k,t_{k+1}}\sum_{\ell=1}^k \alpha_{t_\ell}$ (using $\sum_{i=1}^d \mu_i = 1$ to show that the $\ell=r$ term in the sum produces the desired $-h^1$ term), and so therefore $z \in \Span(C)$, as each $(t_k,t_{k+1}) \in \breve{D}$. This shows the result. Additionally, we note that the choice of $h^1$ to subtract from $\aff(H)$ was arbitrary.
\endproof

\end{appendix}

\bibliographystyle{elsarticle-num}
\bibliography{master.bib}

\end{document}